\documentclass[12pt]{amsart}
\usepackage{amsmath,amssymb,amsfonts,amscd,latexsym,amsthm,mathrsfs,comment}
\usepackage[usenames]{color}
\usepackage{hyperref}
\newtheorem{theorem}{Theorem}
\newtheorem{lemma}{Lemma}
\newtheorem{proposition}{Proposition}
\newtheorem{corollary}{Corollary}
\theoremstyle{definition}

\theoremstyle{remark}
\newtheorem{remark}{Remark}
\let\wt\widetilde

\let\ol\overline
\newcommand\tsqrt[1]{{\textstyle\sqrt{#1}}}
\renewcommand\d{{\mathrm d}}
\newcommand{\m}{\mathrm{m}}
\newcommand{\C}{\mathbb{C}}
\newcommand{\Q}{\mathbb{Q}}
\newcommand{\Z}{\mathbb{Z}}

\renewcommand{\Im}{\operatorname{Im}}

\newcommand{\Res}{\operatorname{Res}}
\definecolor{orange}{rgb}{0.7,0.3,0}

\begin{document}

\title[Boyd's conjecture for a conductor 21 elliptic curve]{Further explorations of Boyd's conjectures\\ and a conductor 21 elliptic curve}

\author{Matilde Lal\'\i n}
\address{D\'{e}partment de Math\'{e}matique et de Statistique, Universit\'{e} de Montr\'{e}al. CP 6128, succ. Centre-Ville. Montreal, QC H3C 3J7, Canada}
\email{mlalin@dms.umontreal.ca}

\author{Detchat Samart}
\address{D\'{e}partment de Math\'{e}matique et de Statistique, Universit\'{e} de Montr\'{e}al. CP 6128, succ. Centre-Ville. Montreal, QC H3C 3J7, Canada}
\address{Department of Mathematics, University of Illinois at Urbana-Champaign, Urbana, IL 61801, USA}
\email{petesamart@gmail.com}

\author{Wadim Zudilin}
\address{School of Mathematical and Physical Sciences, The University of Newcastle, Callaghan, NSW 2308, Australia}
\email{wadim.zudilin@newcastle.edu.au}



\date{31 July 2015. \emph{Revised}: 16 December 2015}
\begin{abstract}
We prove that the (logarithmic) Mahler measure $\m(P)$ of
$P(x,y)=x+1/x+y+1/y+3$ is equal to the $L$-value $2L'(E,0)$ attached
to the elliptic curve $E:P(x,y)=0$ of conductor 21. In order to do
this we investigate the measure of a more general Laurent polynomial
$$
P_{a,b,c}(x,y)=a\biggl(x+\frac1x\biggr)+b\biggl(y+\frac1y\biggr)+c
$$
and show that the wanted quantity $\m(P)$ is related to a
``half-Mahler'' measure of $\tilde P(x,y)=P_{\sqrt7,1,3}(x,y)$. In the
finale we use the modular parametrization of the elliptic curve
$\tilde P(x,y)=0$, again of conductor 21, due to Ramanujan and the
Mellit--Brunault formula for the regulator of modular units.
\end{abstract}

\maketitle

\section{Introduction}\label{intro}

For a nonzero Laurent polynomial $P(x_1,\dots,x_d)\in\mathbb{C}[x_1^{\pm 1},\dots,x_d^{\pm 1}]$, the (logarithmic) Mahler measure of $P$ is defined by
\begin{align*}
\m(P)
&= \frac{1}{(2\pi i)^d}\idotsint\limits_{|x_1|=\dots=|x_d|=1}\log |P(x_1,\dots,x_d)|\frac{\d x_1}{x_1}\dotsb\frac{\d x_d}{x_d}
\\
&= \int_0^1\dotsi\int_0^1\log|P(e^{2\pi i\theta_1},\dots,e^{2\pi i\theta_d})|\d\theta_1\dotsb\d\theta_d.
\end{align*}
It was observed by Deninger~\cite{Deninger} and then systematically verified by Boyd~\cite{Bo98} and Rodriguez-Villegas~\cite{RV} that
general conjectures of Beilinson predict the connection of two-variate Mahler measures $\m(P(x,y))$ with $L$-values of the corresponding curve $P(x,y)=0$,
at least in the cases when the curve is elliptic and the polynomial $P(x,y)=0$ is tempered
(for the definition of a tempered polynomial, see \cite[Section~III]{RV}).
A particular family of such two-variate polynomials, which has launched the elliptic $L$-story, is
$$
P_k(x,y)=x+\frac1x+y+\frac1y+k,
$$
with $k=1$ originally considered by Deninger in~\cite{Deninger} and later extended to $k\in\mathbb Z\setminus\{0\}$ by Boyd
and to $k^2\in\mathbb Z\setminus\{0\}$ by Rodriguez-Villegas (see \cite[Table~1]{Bo98} and \cite[Table~4]{RV}).
The work \cite{RV} has already incorporated some methods for attacking the conjectural evaluations of Mahler measures via $L$-values
and proving several such cases when the related elliptic curves $P(x,y)=0$ have complex multiplication.
Some further development of the techniques in the works of Brunault, Mellit, Rogers and others \cite{Br06,La10,LR,Me12,RZ12,RZ14,Zu14}
has allowed to establish several new conjectural instances when the elliptic curves are parameterized by modular units
(that is, modular functions whose zeroes and poles are only at cusps). The final news is an elegant general formula
of Brunault \cite{Brunault} that allows one to deal with parametrization by Siegel units; this creates an efficient way
to proving any particular Mahler measure evaluation on a case-by-case study, and the work \cite{Brunault} includes numerous
illustrations to the principle.

The original motivation of our project was computing the Mahler measure of polynomial $P_3(x,y)$
and providing explanation of certain related identities first observed numerically.
Boyd \cite[Table 1]{Bo98} conjectured that
\begin{equation}\label{Emain}
\m(P_3)=\frac{21}{2\pi^2}L(E,2)=2L'(E,0),
\end{equation}
where the elliptic curve $E$ of conductor~21 is defined as the zero locus of $P_3(x,y)$,
and this conjecture was confirmed only recently by Brunault \cite{Brunault}.
Note that from the modularity theorem we have $L(E,s)=L(f_{21},s)$, where $f_{21}(\tau)=q-q^2+q^3-q^4-\cdots \in S_2(\Gamma_0(21))$ and $q=e^{2\pi i\tau}$.
The curve $E:P_3(x,y)=0$ does not possess a modular-unit parametrization (and this is the smallest known example, in terms of conductor, of a tempered polynomial
with this property). At the same time, it is isogenous to
$$
\sqrt7\biggl(x+\frac1x\biggr)+y+\frac1y+3=0,
$$
which defines the modular curve $X_0(21)$ and whose modular-unit parametrization was already observed by Ramanujan (see Section~\ref{modular} for details).
However, the polynomial
$$
\sqrt7\biggl(x+\frac1x\biggr)+y+\frac1y+3
$$
does not happen to be tempered, and its Mahler measure is (in a certain sense) trivial\,---\,see equation~\eqref{E:log} below.

The aim of this paper is to develop a link between Mahler measures associated with general 3-parametric polynomials
\begin{equation*}
P_{a,b,c}(x,y)=a\biggl(x+\frac1x\biggr)+b\biggl(y+\frac1y\biggr)+c
\end{equation*}
and the Mahler measure of $\m(P_k)$ and, by these means, to give an alternative proof of evaluation~\eqref{Emain}.
Notice that for a large set of real parameters $a,b,c$ the Mahler measure has a trivial evaluation.

\begin{theorem}
\label{th-log}
If $a,b,c$ are real and $|c|\le2\bigl||a|-|b|\bigr|$ then
\begin{equation*}
\m(P_{a,b,c}(x,y))=\log\max\{|a|,|b|\}.
\end{equation*}
\end{theorem}

The proof of the result makes use of classical Jensen's formula
$$
\m(x-x_0) = \log^+|x_0| = \max\{0, \log |x_0|\};
$$
see Section~\ref{jensen} for details.

To proceed further we observe the symmetry $\m(P_{a,b,c})=\m(P_{b,a,c})$ of the family and the
property $\m(P_{a,b,c})=\log|b|+\m(P_{a/b,1,c/b})$ following from the definition of the Mahler measure.
These features allow us to consider instead the 2-parametric family
\begin{equation*}
P_{a,c}(x,y)=a\biggl(x+\frac1x\biggr)+\biggl(y+\frac1y\biggr)+c,
\end{equation*}
in which $|a|\ge1$. Though several statements about the polynomials remain valid even without the latter condition,
in what follows we focus our study on the case when $a,c$ are real and $a\ge1$, $c>0$.

If we write
\begin{gather*}
yP_{a,c}(x,y)=y^2+A(x)y+B(x)=(y-y_+(x))(y-y_-(x)),
\\
y_{\pm}(x)=\frac{-A(x)\pm\sqrt{A(x)^2-4B(x)}}{2},
\end{gather*}
then it follows from the definition of $\m(P)$ and Jensen's formula that
\begin{equation}\label{E:decom}
\m(P_{a,c}) = \m(yP_{a,c}) = \m(y-y_+) + \m(y-y_-) = \m^{+}(P_{a,c}) + \m^{-}(P_{a,c}),
\end{equation}
where
$$
\m^{\pm}(P_{a,c}) = \m(y-y_{\pm}) = \frac{1}{2\pi i}\int_{|x|=1}\log^+ |y_{\pm}(x)|\,\frac{\d x}{x}.
$$
When $a=\sqrt7$ and $c=3$, it is easily seen from Theorem~\ref{th-log} that
\begin{equation}\label{E:log}
\m(P_{\sqrt{7},3})=\log \sqrt{7}.
\end{equation}
On the other hand, if we decompose $\m(P_{\sqrt{7},3})$ according to~\eqref{E:decom},
each of $\m^{\pm}(P)$ encodes more interesting arithmetic quantities.
Indeed, we have discovered from our numerical computation that
\begin{align}
\m^{-}(P_{\sqrt{7},3})&=\frac{1}{2}L'(f_{21},0)+\frac{3}{8}\log 7,
\label{E:mm}\\
\m^{+}(P_{\sqrt{7},3})&=-\frac{1}{2}L'(f_{21},0)+\frac{1}{8}\log 7,
\nonumber
\end{align}
which will be proven subsequently in this paper. An equivalent form of any of these identities, namely
\begin{equation}\label{E:interm}
\m^{-}(P_{\sqrt{7},3}) - 3\m^{+}(P_{\sqrt{7},3}) = 2L'(f_{21},0),
\end{equation}
is reminiscent of Boyd's conjecture \eqref{Emain}, so we have hypothesized that there is a way to connect $\m(P_3)$
with the $L$-value by means of identity~\eqref{E:interm}.

More generally, we prove in Section~\ref{diff} the following result.

\begin{theorem}\label{th1}
For real $k$, $0<k<4$, let $a=\sqrt{(4+k)/(4-k)}$ and $c=k/\sqrt{4-k}$. Then
\begin{equation}
\m(P_{1,k})=\m^-(P_{a,c})-3\m^+(P_{a,c}).
\label{eq+-}
\end{equation}
\end{theorem}

Note that the parameters participating in~\eqref{eq+-} can be alternatively expressed by means of $a>1$:
\begin{equation*}
c=\frac{\sqrt{2}(a^2-1)}{\sqrt{a^2+1}}
\quad\text{and}\quad
k=\frac{4(a^2-1)}{a^2+1}.
\end{equation*}

In Section~\ref{Weierstrass} we compute a Weierstrass form of the elliptic curve $E$ given by $P_{a,c}(x,y)=0$
and identify some particular (torsion) points on the curve.
The fact that the ``half-Mahler'' measures $\m^{\pm}(P_{a,c})$ happen to be $\mathbb Q$-linear combinations of
$\log a$ and the $L$-value of~$E$ has its roots in a $K$-theoretic interpretation of
the measures and related Beilinson's conjectures. We review the corresponding heuristics, essentially due to
Deninger~\cite{Deninger} and Rodriguez-Villegas \cite{RV}, in Section~\ref{diamond}.
When the parameters $a$ and $c$ are subject to the hypothesis of Theorem~\ref{th1},
the curve $E:P_{a,c}(x,y)=0$ is isogenous to the curve $P_{1,k}(x,y)=0$, and it is this isogeny that gives rise
to the theorem. Section~\ref{isogeny} provides the reader with the details on the isogeny, outlines the proof
of a weaker version of Theorem~\ref{th1} and addresses
some other technical issues that are needed later in Section~\ref{modular}.

Our next principal result here is equation~\eqref{E:mm}.

\begin{theorem}\label{th2}
In the above notation,
\begin{equation*}
\m^{-}(P_{\sqrt{7},3})=\frac{1}{2}L'(f_{21},0)+\frac{3}{8}\log 7.
\end{equation*}
\end{theorem}

Our proof of Theorem~\ref{th2} relies crucially on the fact that the curve $E:P_{\sqrt{7},3}(x,y)=0$ admits a modular-unit parametrization,
which enables us to apply \cite[Theorem~1]{Zu14}. This part is achieved in Section~\ref{modular}.
We then have Boyd's evaluation \eqref{Emain} immediately from Theorems~\ref{th-log}--\ref{th2}. Indeed,
combining the above for $a=\sqrt{7}$ with equation \eqref{E:log} and the modularity theorem, we recover the expected result.

\begin{corollary}
Boyd's conjecture is true for $k=3$. In other words,
\[\m(P_3)= 2L'(E,0),\]
where $E$ is the elliptic curve $P_3(x,y)=0$.
\end{corollary}

Finally, in Section~\ref{final} we outline some natural directions for future study.

\section{Jensen's formula}\label{jensen}

In this section we prove Theorem~\ref{th-log}.

\begin{lemma}
\label{loga}
If $a,c$ are real and $1+|c|/2\le|a|$ then
\begin{equation}
\m(P_{a,c}(x,y))=\m^-(P_{a,c}(x,y))+\m^+(P_{a,c}(x,y))=\log |a|.
\label{eqpm}
\end{equation}
\end{lemma}

\begin{proof}
When $|a|\ge1+|c|/2$, the zeroes of the quadratic polynomial in $x$
$$
\frac{xP_{a,c}(x,y)}a
=x^2+\frac{y+1/y+c}a\,x+1
$$
are complex conjugate of absolute value~1, because $y+1/y$ is real between $-2$ and $2$ for $y$ on the unit circle,
so that the real coefficient $(y+1/y+c)/a$ is in the same range.
Therefore, Jensen's formula applies to imply that the Mahler measure of the latter polynomial
$$
\frac1{2\pi i}\int_{|x|=1}\log\biggl|\frac{xP_{a,c}(x,y)}a\biggr|\,\frac{\d x}x
=\frac1{2\pi i}\int_{|x|=1}\log|P_{a,c}(x,y)|\,\frac{\d x}x-\log|a|
$$
is zero for any $y$; integrating over~$y$ results in the required two-variate Mahler measure evaluation.
\end{proof}

\begin{proof}[Proof of Theorem~\textup{\ref{th-log}}]
It follows from the symmetry with respect to the parameters $a$ and $b$, the relation
$\m(P_{a,b,c})=\log|b|+\m(P_{a/b,1,c/b})$ and Lemma~\ref{loga}.
\end{proof}

\section{Weierstrass form}\label{Weierstrass}

Consider the curve
\[
E_{a,c}:a\biggl(x+\frac{1}{x}\biggr)+y+\frac{1}{y}+c=0.
\]
If we think of $y$ as a function on $x$, we have two solutions given by
\begin{equation}\label{eq:y_pm}
y_\pm(x)=\frac{-(a(x+x^{-1})+c))\pm\sqrt{(a(x+x^{-1})+c))^2-4}}{2}.
\end{equation}
On the torus $|x|=1$ we let $x=e^{i\theta}$, where $-\pi\le\theta\le\pi$, so that $x+1/x=2\cos\theta$. We have $(a(x+1/x)+c)^2-4>0$
if and only if $\cos\theta<t_-$ or $\cos\theta>t_+$, where
$$
t_-=\max\biggl\{-1,-\frac{2+c}{2a}\biggr\} \quad\text{and}\quad t_+=\min\biggl\{1,\frac{2-c}{2a}\biggr\}.
$$
Furthermore, if $\cos\theta<t_-$ then $|y_-(x)|<1<|y_+(x)|$, while if $\cos\theta>t_+$ then $|y_+(x)|<1<|y_-(x)|$.
Denote $\theta_->\theta_+$ the quantities from the interval $[0,\pi]$, for which $\cos\theta_-=t_-$ and $\cos\theta_+=t_+$.
In the notation of Section~\ref{intro} we have
\begin{align}
\label{m-}
\m^-(P_{a,c})
&=\frac1{2\pi}\int_{-\theta_+\le\arg x\le\theta_+}\log|y_-(x)|\,\d\arg x
=\frac1\pi\int_{0\le\arg x\le\theta_+}\log|y_-(x)|\,\d\arg x
\\
&=\frac1\pi\int_{t_+}^1\frac{\log|at+c/2+\sqrt{(at+c/2)^2-1}|}{\sqrt{1-t^2}}\,\d t,
\nonumber\\ 
\label{m+}
\m^+(P_{a,c})
&=\frac1{2\pi}\int_{\theta_-\le\arg x\le2\pi-\theta_-}\log|y_+(x)|\,\d\arg x
=\frac1\pi\int_{\theta_-\le\arg x\le\pi}\log|y_+(x)|\,\d\arg x
\\
&=\frac1\pi\int_{-1}^{t_-}\frac{\log|at+c/2-\sqrt{(at+c/2)^2-1}|}{\sqrt{1-t^2}}\,\d t.
\nonumber
\end{align}
Note that for $a=1$ we get $\theta_-=\pi$, so that $\m^+(P_{1,c})=0$ and
$$
\m(P_{1,c})=\m^-(P_{1,c})=\frac1\pi\int_{0\le\arg x\le\theta_+}\log|y_-(x)|\,\d\arg x.
$$

By taking
\[
X=-\frac{a}{xy}, \qquad Y=\frac{a}{2xy}\left(y-\frac{1}{y}-a\left(x-\frac{1}{x}\right)\right),
\]
we obtain the Weierstrass form
\[
Y^2=X\left(X^2+\left(\frac{c^2}{4}-1-a^2\right)X+a^2\right).
\]
We record here the inverse transformations
\begin{equation}\label{eq:Wtransf}
x=\frac{a(cX-2Y)}{2X(X-a^2)}, \qquad y=\frac{cX+2Y}{2X(X-1)}.
\end{equation}

In this Weierstrass form we have two points $P=(1,c/2)$ and $Q=(a^2,ca^2/2)$ that satisfy
\begin{equation}
\begin{aligned}
2P&=\biggl(\frac{(a^2-1)^2}{c^2},\frac{(a^2-1)(2a^4-a^2c^2-4a^2-c^2+2)}{2c^3}\biggr),
\\
2Q&=\biggl(\frac{(a^2-1)^2}{c^2},-\frac{(a^2-1)(2a^4-a^2c^2-4a^2-c^2+2)}{2c^3}\biggr),
\end{aligned}
\label{2PQ}
\end{equation}
and $P+Q=(0,0)$. Notice that when $a=1$, we get $P=Q$ and the point has order~4.

The images of $P$ and $Q$ on the original curve $E_{a,c}:P_{a,c}(x,y)=0$ are written as
$$
\bar P=(0,\infty) \quad\text{and}\quad \bar Q=\biggl(\frac{1-a^2}{ac},\frac c{a^2-1}\biggr)
$$
if $a\ne1$, and $\bar P=\bar Q=(0,\infty)$ if $a=1$.

If $c/2<a+1$, the points
\begin{equation}
\begin{aligned}
\bar S_\pm&=\biggl(\frac{1-c/2\pm\sqrt{(1-c/2)^2-a^2}}a,-1\biggr),
\\
\bar T_\pm&=\biggl(\frac{-1-c/2\pm\sqrt{(1+c/2)^2-a^2}}a,1\biggr)
\end{aligned}
\label{pointsST}
\end{equation}
on the curve $E_{a,c}$ correspond to $\arg x=\pm\theta_+$ and $\arg x=\pm\theta_-$ respectively; the latter case is considered only if $c/2<a-1$.
On the curve in its Weierstrass form the points are
\begin{align*}
S_\pm&=\textstyle\bigl(1-c/2\mp\sqrt{(1-c/2)^2-a^2},\pm\sqrt{(1-c/2)^2-a^2}\bigl(1-c/2\mp\sqrt{(1-c/2)^2-a^2}\bigr)\bigr),
\\
T_\pm&=\textstyle\bigl(1+c/2\pm\sqrt{(1+c/2)^2-a^2},\pm\sqrt{(1+c/2)^2-a^2}\bigl(1+c/2\pm\sqrt{(1+c/2)^2-a^2}\bigr)\bigr).
\end{align*}
Then we have
$$
2S_\pm=2T_\pm=P
\quad\text{and}\quad
S_++S_-=T_++T_-=-Q.
$$

\begin{remark}
\label{rem-specials}
At this stage we would like to identify two particular choices for $a$ and $c$ that correspond
to special structures of the group generated by $P$ and~$Q$.

First, when
\begin{equation}
c=\frac{\sqrt2(a^2-1)}{\sqrt{a^2+1}},
\label{c0}
\end{equation}
and in this case only, we have
$S_++T_+=-P$.
This case can be alternatively characterized by
\begin{equation*}
2P=2Q=\biggl(\frac{a^2+1}2,0\biggr)
\end{equation*}
in accordance with \eqref{2PQ}.

Second, if we prescribe
\begin{equation}\label{torsion6}
c=a^2-1
\end{equation}
then
\begin{equation*}
3P=O, \quad 3Q=(0,0).
\end{equation*}
\end{remark}

As usual, for any meromorphic functions $f,g$ on a smooth projective curve $C$, we define
\begin{equation}
\label{etadef}
\eta(f,g)=\log |f| \,\d\arg(g) - \log |g|\,\d\arg(f),
\end{equation}
where $\d \arg(g)$ is globally defined as $\Im(\d g/g)$.

Considering $x$ and $y=y_\pm$ as rational functions on the curve in its Weierstrass form we get from the earlier formulas via
the change of variables, for $c/2<a+1$,
\begin{align}
\label{yy-}
\m^-(P_{a,c})
&=\frac1{2\pi}\int_{[S_-,S_+]}\eta(y_-,x)
=-\frac1{2\pi}\int_{[S_-,S_+]}\eta(x,y_-)
\\ \intertext{and also, when $c/2<a-1$,}
\label{yy+}
\m^+(P_{a,c})
&=\frac1{2\pi}\int_{[T_+,T_-]}\eta(y_+,x)
=-\frac1{2\pi}\int_{[T_-,T_+]}\eta(x,y_-).
\end{align}

We stress on the fact that here and in what follows the integration of $\eta(x,y)$ is performed on the curve in its Weierstrass form
(that is, in the coordinates $(X,Y)$); the coordinates $x$ and $y$ are understood as the rational functions \eqref{eq:Wtransf} of $X$ and $Y$.
This implies that we integrate in \eqref{yy-} and \eqref{yy+} over the intervals $[S_-,S_+]$ and $[T_-,T_+]$ rather than over $[\bar S_-,\bar S_+]$ and $[\bar T_-,\bar T_+]$.

There is no simple recipe to identify $y$ (as defined by equation \eqref{eq:Wtransf}) with either $y_-$ or~$y_+$.
However we can notice, on the basis of equations \eqref{eq:y_pm}, that
as $x \to 0$ and $x\to\infty$ we have $y_+\to 0$ and $y_-\to \infty$, respectively,
and this can be used in the identification.
We discuss these issues in Section~\ref{isogeny} below.

\section{Elliptic regulator}
\label{diamond}

This section has two goals. The first goal is to cast the (``half-'') Mahler measures of $P_{a,c}(x,y)$ in terms of elliptic regulators
and present some evidence for the former to be rationally related to the $L$-value of the underlying curve $E_{a,c}:P_{a,c}(x,y)=0$.
The second goal has a more technical favour: computing certain tame symbols associated with the integral that appears in Theorem~\ref{th2}.
This computation is an essential part of our proof of the theorem in Section~\ref{modular}. More specifically, the tame symbols will dictate
the choice of the integration path in Lemma~\ref{lem2/7}.

Throughout the section we take $a$ and $c$ such that $a^2,c^2 \in \Q$ so that the curve $E_{a,c}$ is defined over~$\Q$.

First we recall the relationship of the Mahler measure to the regulator and tame symbols.

Matsumoto's theorem yields a simple expression for the second $K$-group of a field $F$:
\begin{align*}
K_2(F)
&\cong F^\times\otimes_{\Z}F^\times/\langle x\otimes(1-x):x\in F,\ x\ne0,1\rangle
\\
&\cong \wedge^2 F^\times /\langle x\otimes(1-x):x\in F,\ x\ne0,1\rangle.
\end{align*}
For a discrete valuation $v$ and a maximal ideal $\mathcal{M}$,
the tame symbol of $\{x,y\}\in K_2(F)$ at $v$ is given by
\[
(x,y)_v \equiv (-1)^{v(x)v(y)} \frac{x^{v(y)}}{y^{v(x)}} \bmod \mathcal{M}
\]
(see \cite{RV}).

Let $E/\Q$ be an elliptic curve. In the case when $F=\mathbb{Q}(E)$, each point $S \in E(\bar{\mathbb{Q}})$
determines a valuation by considering the order of the rational functions at~$S$. Denote this valuation by $v_S$.
An element $\{x,y\} \in K_2(\mathbb{Q}(E))\otimes\mathbb{Q}$ can be seen as an element in $K_2(E)\otimes\mathbb{Q}$
whenever $(x,y)_{v_S}=1$ for all $S \in E(\bar{\mathbb{Q}})$. This is the case when we consider the symbol $\{x,y\}$
in $K_2(\Q(E_{a,c}))$ when $a=1$, as this is a tempered family by a result from~\cite{RV} and therefore it satisfies the triviality of tame symbols.
However, this is not the case for the symbol $\{x,y\}$ in $K_2(\Q(E_{a,c}))$ for general $a$.
In fact, we have the following result which will be shown later in this section.

\begin{lemma}\label{tame}
On the curve $E_{a,c}$ we have
\[
|(x,y)_P|=|(x,y)_{-Q}|=\frac{1}{a},
\qquad
|(x,y)_{P+Q}|=|(x,y)_O|=a;
\]
all the other tame symbols $(x,y)_R$ are trivial.
\end{lemma}

The regulator map of Bloch and Beilinson \cite{Be,B} may be defined by
\[
r \colon K_2(E) \to H^1(E,\mathbb{R}),
\quad
\{x,y\}\mapsto \left\{ \gamma \to \int_\gamma \eta(x,y)\right \},
\]
for $\gamma \in H_1(E,\mathbb{Z})$.

\begin{remark}
\label{rem-K2}
Technically, the regulator is defined over $K_2(\mathcal{E})$ where $\mathcal{E}$ is the N\'eron model of the elliptic curve~$E$;
$K_2(\mathcal{E})$ may be considered, up to torsion, as a subgroup of $K_2(E)$ determined by finitely many extra conditions given
by the primes of bad reduction of~$E$ (see~\cite{BG}).
\end{remark}

When $a=1$, after the work of Deninger \cite{Deninger} and Rodriguez-Villegas \cite{RV} we can write
\begin{equation}
\m(P_{1,c})=-\frac{1}{2\pi}r_c(\{x,y\})([\gamma]),
\label{mP1c}
\end{equation}
where $r_c$ denotes the regulator on $E_{1,c}$ and $\gamma$ is a generator of $H_1(E_{1,c},\Z)^-$, the homology group generated by
the cycles that change signs under complex conjugation of the coordinates $x$ and $y$, in this case given by the cycle $|x|=1$, $|y_-|\ge1$.

Let $\Z[E(\C)]^-$ be the group of divisors on $E$ modulo divisors of the form $(a)+(-a)$ for $a \in E(\C)$.
To two rational functions $f,g \in \C(E)^\times$ with divisors $(f)=\sum_i m_i(a_i)$ and $(g)=\sum_j n_j(b_j)$ the diamond operator
\[
\diamond\colon \wedge^2 \C(E)^\times \to [\Z(E(\C)]^-
\]
assigns the divisor
\[
(f)\diamond (g) =\sum_{i,j} m_in_j(a_i-b_j).
\]
The elliptic dilogarithm is a certain Eisenstein--Kronecker series that satisfies
\[
\int_\gamma \eta(f,g)=D_E\bigl((f)\diamond (g)\bigr),
\]
where $\gamma$ is a generator of $H_1(E,\Z)^-$; details can be found, for example, in \cite{Deninger, LR, RV}.
For the purpose of this paper we only use that the elliptic dilogarithm is a function of $E$ and $(f)\diamond (g)$, in the same spirit as in \cite{RV2}.

As explained in \cite[Lemma on p.~25]{RV} we have that
\[
\Res_v\eta(f,g)=\log|(f,g)_v|.
\]
Thus, these residues contribute to the integrals in the general case of $E_{a,c}$, where some tame symbols are not trivial.
Following the argument similar to that of \cite[pp.~272--274]{Deninger} one sees
that the paths of integration $[S_-,S_+]$ and $[T_-,T_+]$ can again be interpreted in terms of elements in $H_1(E_{a,c},\Q)$.
To show that they are in $H_1(E_{a,c},\Q)^-$, consider the form
\[
\omega=\frac{\d X}{2Y}=-\frac{\d x}{2x(y-y^{-1})}.
\]
{}From the construction of path $[S_-,S_+]$ we see that $\int_{[S_-,S_+]} \omega$ is purely imaginary, thus $[S_-,S_+]\in H_1(E_{a,c},\Q)^-$.
The same reasoning applies to the path $[T_-,T_+]$. Furthermore, assuming that
$S_\pm$ and $T_\pm$ are points of order $N$, we have the arithmetically stronger versions
$[S_-,S_+],[T_-,T_+]\in H_1(E_{a,c},(1/N)\Z)^-$.

Therefore, we have the following observation.

\begin{lemma}
\label{lem3}
Let $f, g \in \Q(E_{a,c})$ be such that any valuation of their tame symbol is an integral \textup(possibly zero\textup) power of~$a$.
Then there exist rational numbers $p$ and $q$ such that
\[
\frac1{2\pi}\int_{[S_-,S_+]}\eta(f,g) = \frac{q}{2\pi}\,D_{E_{a,c}}\bigl((f)\diamond (g)\bigr)+p\,\log a,
\]
where $q$ does not depend on the choice of $f,g$, and the similar statement \textup(with possibly different $p$ and $q$\textup)
is true for the integral over $[T_-,T_+]$.

Furthermore, if $S_\pm, T_\pm$ are points of order $N$ then both $Np$ and $Nq$ are integers.
\end{lemma}

To compute the divisors $(x)$ and $(y)$, we first perform the computation of the divisors of some rational
functions in the Weierstrass form of the curve:
\begin{align*}
(X)&=2(P+Q)-2O,\\
(X-1)&=(P)+(-P)-2O,\\
(X-a^2)&=(Q)+(-Q)-2O,\\
(cX+2Y)&=(P+Q)+(-Q)+(-P)-3O,\\
(cX-2Y)&=(P+Q)+(Q)+(P)-3O.
\end{align*}
By referring to equations \eqref{eq:Wtransf} we obtain
\[
(x)=-(P+Q)+(P)-(-Q)+O \quad\text{and}\quad (y)=-(P+Q)-(P)+(-Q)+O;
\]
thus, the diamond operation of the divisors is given by
\begin{equation}\label{xy}
(x)\diamond(y)=4(P)+4(Q).
\end{equation}

The knowledge of the divisors $(x)$ and $(y)$ allows us to compute the tame symbols.

\begin{proof}[Proof of Lemma~\textup{\ref{tame}}]
It suffices to concentrate on the tame symbols over points that
are supported in the divisors of $x,y$, since the tame symbol is trivial for the other points:
\begin{align*}
|(x,y)_P|&=\left|\frac{1}{xy}\right|_P=\left|\frac{X}{a}\right|_{P}=\frac{1}{a},\\
|(x,y)_{-Q}|&=\left|xy\right|_{-Q}=\left|\frac{a}{X}\right|_{-Q}=\frac{1}{a},\\
|(x,y)_{P+Q}|&=\left|\frac{y}{x}\right|_{P+Q}=\left|\frac{(cX+2Y)(X-a^2)}{a(cX-2Y)(X-1)}\right|_{P+Q}=a,\\
|(x,y)_O|&=\left|\frac{x}{y}\right|_O= \left|\frac{a(cX-2Y)(X-1)}{(cX+2Y)(X-a^2)}\right|_{O}= a.
\qedhere
\end{align*}
\end{proof}

We can now summarize the outcomes of Lemmas~\ref{tame} and~\ref{lem3} in the following form.

\begin{proposition}\label{prop:reg}
Under the rationality conditions $a^2,c^2\in\Q$ for $a,c$, there exist $p_1,q_1,p_2,q_2\in\Q$ such that
\begin{align*}
\frac1{2\pi}\int_{[S_-,S_+]}\eta(x,y) &= \frac{q_1}{2\pi}\,D_{E_{a,c}}\bigl((x)\diamond(y)\bigr)+p_1\log a,
\\
\frac1{2\pi}\int_{[T_-,T_+]}\eta(x,y) &= \frac{q_2}{2\pi}\,D_{E_{a,c}}\bigl((x)\diamond(y)\bigr)+p_2\log a.
\end{align*}

Furthermore, if $S_\pm, T_\pm$ are points of order $N$ then $Np_1,Nq_1,Np_2,Nq_2\in\Z$.
\end{proposition}

Still without identifying the rational function $y$ on $E_{a,c}$ with either $y_-$ or $y_+$,
we see that the integrals in Proposition~\ref{prop:reg} are, up to sign, the ``half-Mahler'' measures
$\m^-(P_{a,c})$ and $\m^+(P_{a,c})$. Therefore, the result can be interpreted as follows:
each of the measures is a $\Q$-linear combination of the elliptic dilogarithm $D_{E_{a,c}}((x)\diamond(y))/(2\pi)$ and $\log a$.
At the same time, the conjectures of Bloch and Beilinson \cite{Be,B} predict that $K_2(E_{a,c})$ has rank one and, thus,
the dilogarithmic quantity is a rational multiple of $L'(E_{a,c},0)$,
so that both $\m^-(P_{a,c})$ and $\m^+(P_{a,c})$ are expected to be $\Q$-linear combinations of the $L$-value and $\log a$.

\section{Isogeny}\label{isogeny}

In this section we concentrate on the curve $E_{a,c}:P_{a,c}(x,y)=0$ subject to the condition \eqref{c0}. Its Weierstrass form is
$$
Y^2=X\biggl(X-\frac{2a^2}{a^2+1}\biggr)\biggl(X-\frac{a^2+1}2\biggr).
$$
One can construct a degree 2 isogeny $\varphi$ between the latter and
\begin{equation}
{Y'}^2=X'\biggl({X'}^2+\frac{2(a^4-6a^2+1)}{(a^2+1)^2}X'+1\biggr)
\label{cv}
\end{equation}
given by
$$
X'=\frac{2(a^2+1)Y^2}{((a^2+1)X-2a^2)^2}, \quad Y'=\frac{2^{3/2}Y}{(a^2+1)^{3/2}}\biggl(1+\frac{a^2(a^2-1)^2}{((a^2+1)X-2a^2)^2}\biggr).
$$
Note that \eqref{cv} is the Weierstrass form of Boyd's curve $P_{1,k}(x',y')=0$, where $k=4(a^2-\nobreak1)/\allowbreak(a^2+\nobreak1)$.
We record here some relevant relations between the points that we have discussed in Section~\ref{Weierstrass}:
\[
\varphi(P)=\varphi(Q)=P'=2S_\pm'
\]
and
\[
\varphi(S_+)=\varphi(T_+)=S_+', \qquad \varphi(S_-)=\varphi(T_-)=S_-'.
\]

Our principal task in this section is to prove a weaker version of Theorem~\ref{th1}.
This proof enriches us with crucial understanding of tame symbols that will be needed in Section~\ref{modular}
as well as clarifies the relationship between $y$ and~$y_\pm$.  Theorem~\ref{th1} will be proven completely in Section \ref{diff}.

\begin{proposition}\label{prop:regul}
Under the above rationality conditions for $a,c$, together with \eqref{c0}, there exists $p\in(1/8)\Z$ such that
\begin{align}
\label{etaS}
\frac1{2\pi}\int_{[S_-,S_+]}\eta(x,y)&=\frac{1}{8\pi}\int_{[S_-',S_+']} \eta(x',y')+p\,\log a,
\\
\label{etaT}
\frac1{2\pi}\int_{[T_-,T_+]}\eta(x,y)&= \frac{1}{8\pi}\int_{[S_-',S_+']} \eta(x',y')+(p-1)\,\log a.
\end{align}
\end{proposition}

In what follows we only discuss the equality \eqref{etaS} as the second one is the immediate consequence of
the expressions \eqref{yy-}, \eqref{yy+} and Lemma~\ref{loga}.
The apparent sign discrepancy between the above equations and the expected relationships
with $\m^\pm(P_{a,c})=\m(y-y_\pm)$ lies in the fact that we have $y$ instead of $y_\pm$. To go from $y$ to
$y_-$ one must either take $y_-=y$ or $y_-=1/y$ and they are different choices for $[S_-,S_+]$
and $[T_-,T_+]$; see Remark~\ref{rem 2} later in this section.

\begin{remark}
\label{remm}
The expression \eqref{mP1c} for $\m(P_{a,c})$ when $a=1$ as well as formulas \eqref{yy-} and \eqref{yy+} show that
Proposition~\ref{prop:regul} is indeed a weaker version of Theorem~\ref{th1}: the rational $p$ is specified in the latter to be~$3/4$.
The integrality of~$8p$ guaranteed by Proposition~\ref{prop:regul} gives, in fact, a practical recipe to compute the number by providing
a simple numerical approximation to~$p$; however the latter has to be done for each particular choice of $a,c$\,---\,the proposition
does not guarantee that the integer~$p$ is independent of the parameters.
\end{remark}

Before proceeding with the proof of Proposition \ref{prop:regul} we recall the following simple fact about
the image of tame symbols under an isogeny.

\begin{lemma}
\label{lll}
Let $E$ and $E'$ be elliptic curves over $\Q$ and $\varphi\colon E\to E'$ an isogeny.
Let $R$ be a point of $E(\C)$ and $R'=\varphi(R)$ its image on~$E'$. For $x,y \in \Q(E')$,
we have $x\circ \varphi, y\circ \varphi \in \Q(E)$ and
\[
(x\circ \varphi,y\circ \varphi)_R=(x,y)_{R'}.
\]
\end{lemma}

\begin{proof}
By definition,
\begin{align*}
(x\circ \varphi,y\circ \varphi)_R
&=(-1)^{v_R(x \circ \varphi) v_R(y \circ \varphi)} \frac{(x\circ \varphi)^{v_R(y\circ \varphi)}}{(y\circ \varphi)^{v_R(x\circ \varphi)}}\bigg|_R
\\ \intertext{(using $v_R(x \circ \varphi)=v_{R'}(x)$ and $v_R(y \circ \varphi)=v_{R'}(y)$)}
&=(-1)^{v_{R'}(x) v_{R'}(y)}\frac{(x\circ \varphi)^{v_{R'}(y)}}{(y\circ \varphi)^{v_{R'}(x)}}\bigg|_R
\\
&=(-1)^{v_{R'}(x) v_{R'}(y)}\frac{(x)^{v_{R'}(y)}}{(y)^{v_{R'}(x)}}\bigg|_{R'}
=(x,y)_{R'}.
\qedhere
\end{align*}
\end{proof}

\begin{corollary} \label{tame-varphi}
Under hypothesis \eqref{c0},
the tame symbols $(x'\circ \varphi,y'\circ \varphi)_v$ are trivial and $\{x'\circ \varphi,y'\circ \varphi\} \in K_2(E_{a,c})$.
\end{corollary}

\begin{proof}
The triviality of the tame symbol $\{x',y'\}$ in $K_2(\Q(E_{1,k}))$, where $k=4(a^2-1)/(a^2+1)$, follows from the fact
that the polynomial $P_{1,k}$ is tempered (see Section~\ref{diamond}). It remains to apply the isogeny $\varphi$ and Lemma~\ref{lll}.
\end{proof}

\begin{proof}[Proof of Proposition~\textup{\ref{prop:regul}}]
Condition~\ref{c0} and the computation in Section~\ref{Weierstrass} imply that the points $S_{\pm},T_{\pm}$ all have order~8.
As shown in Section~\ref{diamond}, the divisors of $(x)$ and $(y)$ are supported at the points $P$, $-Q$, $P+Q$ and $O$ on the curve $E_{a,c}$,
while their diamond operation is given in~\eqref{xy}.

Since the isogeny $\varphi$ sends the endpoints of the paths $[S_-,S_+]$ and $[T_-,T_+]$ in $E_{a,c}$ to
the endpoints of $[S_-',S_+']$ in $E_{1,k}$ and the tame symbols
$(x',y')_R$ and $(x'\circ \varphi, y'\circ \varphi)_R$ are trivial, we can write
\begin{equation*}
\begin{aligned}
\frac1{2\pi}\int_{[S_-,S_+]}\eta(x'\circ\varphi,y'\circ\varphi)&=\frac1{2\pi}\int_{[S_-',S_+']} \eta(x',y'),
\\
\frac1{2\pi}\int_{[T_-,T_+]}\eta(x'\circ\varphi,y'\circ\varphi)&=\frac1{2\pi}\int_{[S_-',S_+']} \eta(x',y').
\end{aligned}
\end{equation*}
In addition, applying the isogeny we discover that
\begin{equation*} 
(x'\circ\varphi)\diamond(y'\circ\varphi)=16(P)+16(Q)=4(x)\diamond(y).
\end{equation*}
Using now Lemma~\ref{lem3} we obtain
\begin{align*}
\frac1{2\pi}\int_{[S_-,S_+]}\eta(x,y)
&=\frac{q}{2\pi}\,D_{E_{a,c}}\bigl((x)\diamond(y)\bigr)+p\,\log a
\nonumber\\\
&=\frac{q}{2\pi}\cdot\frac14\,D_{E_{a,c}}\bigl((x'\circ \varphi)\diamond(y'\circ \varphi)\bigr)+p\,\log a
\nonumber\\
&=\frac{q}{8\pi}\int_{[S_-,S_+]}\eta(x'\circ\varphi,y'\circ\varphi)+p\,\log a
\nonumber\\
&=\frac{1}{8\pi}\int_{[S_-',S_+']}\eta(x',y')+ p\,\log a
\end{align*}
with $p\in(1/8)\Z$.
As noticed earlier, equality \eqref{etaT} follows from \eqref{etaS} and Lemma~\ref{loga}.
\end{proof}

\begin{remark}
\label{rem 2}
As observed at the end of Section~\ref{Weierstrass}, as $x \to 0$ and $x\to \infty$,
we have $y_+\to 0$ and $y_-\to \infty$, respectively. Here we see that the condition
$x\to0,\infty$ is satisfied over $P, -Q, P+Q$ and $O$. However, we see that $y \to 0$ for $-Q$
and $O$ while $y \to \infty$ for $P$ and $P+Q$. Because of this, we cannot establish
a connection between $y$ and $y_\pm$ for the entire set of points in $E_{a,c}$.
As defined in equations \eqref{pointsST}, $S_\pm$ and $T_\pm$ have $y$-coordinate with absolute value~1,
which in principle can be interpreted as either $y_+$ or $y_-$.
Combining the equations in Proposition~\ref{prop:regul}
shows that $y=y_+$ over $[S_-,S_+]$ while $y=y_-$ over $[T_-,T_+]$.
\end{remark}

\section{Relations of elliptic integrals}\label{diff}

In this section we establish Theorem~\ref{th1} by differentiation with respect to the \emph{continuous} real parameter $a>1$.

The following two auxiliary results about elliptic integrals precede our proof.

\begin{lemma}
\label{lem-EI1}
For $v$ real, $1<v<\sqrt2$, the following equality is valid:
\begin{align*}
&
\int_{3-2v^2}^1\frac{\d T}{\sqrt{(1-T^2)(T+2v^2-1)(T+2v^2-3)}}
\\ &\qquad
=\int_{(\sqrt{2-v^2}-v^2+1)/v}^1\frac{\d t}{\sqrt{(1-t^2)(v^2t^2+2(v^2-1)vt+v^4-v^2-1)}}.
\end{align*}
\end{lemma}

\begin{proof}
By $L$ and $R$ denote the integrals on the left- and right-hand sides of the identity to be shown.
We will make use of the quadratic transformation \cite[Eq.~(3.1.11)]{AAR}
\begin{equation*}
F\biggl(\frac{4z}{(1+z)^2}\biggr)=(1+z)F(z^2)
\end{equation*}
of the hypergeometric function
\begin{equation*}
F(z)={}_2F_1\biggl(\begin{matrix}\frac{1}{2},\,\frac{1}{2}\\1\end{matrix}\biggm|z\biggr)
=\sum_{n=0}^\infty{\binom{2n}n}^2\biggl(\frac z{16}\biggr)^n
=\frac{2}{\pi}\int_0^1\frac{\d x}{\sqrt{(1-x^2)(1-zx^2)}}.
\end{equation*}

The substitution
\begin{equation*}
T=\frac{(v^2-1)x^2+3-2v^2}{(1-v^2)x^2+1}
\end{equation*}
in the integral defining $L$ translates it into
\begin{equation*}
L=\int_0^1\frac{\d x}{\sqrt{(1-x^2)(1-(v^2-1)^2x^2)}}
=\frac\pi2\,F\bigl((v^2-1)^2\bigr).
\end{equation*}
Similarly, the substitution
\begin{equation}
t=\frac{\beta(\alpha-1)x^2+\alpha(1-\beta)}{(\alpha-1)x^2+1-\beta},
\label{E:subs}
\end{equation}
with
$$
\alpha=\frac{\sqrt{2-v^2}-v^2+1}{v} \quad\text{and}\quad \beta=\frac{-\sqrt{2-v^2}-v^2+1}{v},
$$
in the integral for $R$ results in
\begin{align*}
R&=\frac2{\sqrt{1+2v\sqrt{2-v^2}+2v^2-v^4}}\int_{0}^1\frac{\d x}{\sqrt{(1-x^2)(1-w^2x^2)}}
\\
&=\frac{\pi F(w^2)}{\sqrt{1+2v\sqrt{2-v^2}+2v^2-v^4}},
\end{align*}
where
$$
w=\frac{1-2v\sqrt{2-v^2}+2v^2-v^4}{(v^2-1)^2}.
$$
Since
$$
\frac{4w}{(1+w)^2}=(v^2-1)^2 \quad\text{and}\quad 1+w=\frac{2}{\sqrt{1+2v\sqrt{2-v^2}+2v^2-v^4}},
$$
it follows from the quadratic transformation given above that $R=L$.
\end{proof}

\begin{lemma}
\label{lem-EI2}
For $v$ real, $1<v<\sqrt2$, we have
\begin{equation*}
v\int_{(\sqrt{2-v^2}-v^2+1)/v}^1\frac{(2t+v)\,\d t}{\sqrt{(1-t^2)(v^2t^2+2(v^2-1)vt+v^4-v^2-1)}}
=\frac{3\pi}2.
\end{equation*}
\end{lemma}

\begin{proof}
Our goal is to convert the elliptic integral on the left-hand side
into its Legendre canonical form, that is, to express it through the standard elliptic integrals
\begin{gather*}
K(z)=\int_0^1\frac{\d x}{\sqrt{(1-x^2)(1-z^2x^2)}}, \qquad
E(z)=\int_0^1\frac{\sqrt{1-z^2x^2}}{\sqrt{1-x^2}}\d x,
\\
\Pi(n,z)=\int_0^1\frac{\d x}{(1-nx^2)\sqrt{(1-x^2)(1-z^2x^2)}},
\end{gather*}
and then to compute the derivative. The complete elliptic integrals are hypergeometric functions,
$\Pi(n,z)$ is a two-variable Appell $F_1$ hypergeometric function, and it is classically known that all their (partial) derivatives
are given by means of themselves \cite[Chapter~19]{OLBC}. In particular,
\begin{align*}
\frac{\d}{\d z}K(z)&=\frac{1}{z(1-z^2)}E(z)-\frac{1}{z}K(z),\\
\frac{\partial}{\partial n}\Pi(n,z)&= \frac{nE(z)+(z^2-n)K(z)+(n^2-z^2)\Pi(n,z)}{2n(n-1)(z^2-n)},\\
\frac{\partial}{\partial z}\Pi(n,z)&= \frac{z}{(z^2-1)(n-z^2)}(E(z)+(z^2-1)\Pi(n,z)).
\end{align*}

Denote the integral on the left-hand side of the required equality by~$L$ and employ the notation from the proof of Lemma~\ref{lem-EI1}.
We again start with the substitution \eqref{E:subs} to express $L$ as
\begin{align*}
L
&=\frac{2v}{\sqrt{1+2v\sqrt{2-v^2}+2v^2-v^4}}
\int_{0}^1\biggl(2\biggl(\beta+\frac{(\alpha-\beta)(\beta-1)}{(\beta-1)-(\alpha-1)x^2}\biggr)+v\biggr)\frac{\d x}{\sqrt{(1-x^2)(1-w^2x^2)}}
\\
&=\frac{2v}{\sqrt{1+2v\sqrt{2-v^2}+2v^2-v^4}}\biggl((2\beta+v)K(w)+2(\alpha-\beta)\Pi\biggl(\frac{\alpha-1}{\beta-1},w\biggr)\biggr).
\end{align*}
Since $v$ ranges between $1$ and $\sqrt{2}$, it is immediate that $w\in (0,1)$.
Letting $r=\sqrt{w}$, we can manipulate the last expression above to get
\begin{equation}
L=(1-r)\bigl((1-r-2\tsqrt{r^2+1})K(r^2)+4\tsqrt{r^2+1}\,\Pi(f(r),r^2)\bigr),
\label{E:I1}
\end{equation}
where $f(r)=r(\sqrt{r^2+1}+1)(\sqrt{r^2+1}-r)$.
To compute the derivative of $L$ on the interval $(0,1)$ we apply the chain rule and the above formulas for differentiation:
\begin{align*}
\frac{\d}{\d r}K(r^2)
&=\frac{2}{r(1-r^4)}E(r^2)-\frac{2}{r}K(r^2),\\
\frac{\d}{\d r}\Pi(f(r),r^2)
&=\left(\frac{f'(r)}{2(f(r)-1)(r^4-f(r))}+\frac{2r^3}{(r^4-1)(f(r)-r^4)}\right)E(r^2)
\\ &\kern-10mm
+\frac{f'(r)}{2f(r)(f(r)-1)}K(r^2)+\left(\frac{(f(r)^2-r^4)f'(r)}{2f(r)(f(r)-1)(r^4-f(r))}+\frac{2r^3}{f(r)-r^4}\right)\Pi(f(r),r^2).
\end{align*}
This tedious computation results in $\d L/\d r=0$, hence $L$ must be constant on $(0,1)$.

By taking the limits of the both sides in~\eqref{E:I1} as $r\to0^+$, one sees immediately that
\begin{equation*}
L=-K(0)+4\Pi(0,0)=-\frac{\pi}{2}+\frac{4\pi}{2}=\frac{3\pi}{2},
\end{equation*}
as required.
\end{proof}

\begin{proof}[Proof of Theorem~\textup{\ref{th1}}]
Since this is true as $a\to1^+$, it is sufficient to establish the equality of the derivatives of both sides of~\eqref{eq+-} with respect to~$a$.

Under the constraint \eqref{c0} we take
$$
a=\frac{u^2+2u-1}{u^2-2u-1}, \quad
\frac c2=\frac{a^2-1}{\sqrt{2(a^2+1)}}=\frac{4u(u^2-1)}{(u^2+1)(u^2-2u-1)},
\qquad\text{where}\quad u>\sqrt2+1,
$$
so that
$$
t_-=-\frac{1+c/2}a=-\frac{u^4+2u^3-6u-1}{(u^2+1)(u^2+2u-1)},
\quad
t_+=\frac{1-c/2}a=\frac{u^4-6u^3+2u-1}{(u^2+1)(u^2+2u-1)},
$$
and differentiate \eqref{m-}, \eqref{m+} with respect to the parameter $u$ (also using that
$|y_-(t_+)|=|y_+(t_-)|=1$):
\begin{align*}
\frac{\partial\m^-(P_{a,c})}{\partial u}
&=-\frac{(u^2+1)^3J_-+(u^2+2u-1)(u^4-2u^3+2u^2+2u+1)I}{(u^2+1)^2(u^4-6u^2+1)},
\\
\frac{\partial\m^+(P_{a,c})}{\partial u}
&=\frac{(u^2+1)^3J_++(u^2+2u-1)(u^4-2u^3+2u^2+2u+1)I}{(u^2+1)^2(u^4-6u^2+1)},
\end{align*}
where
\begin{gather*}
J_-=\frac4\pi\int_{t_+}^1\frac{t\,\d t}{\sqrt{(1-t^2)(t-t_-)(t-t_+)}},
\qquad
J_+=\frac4\pi\int_{-1}^{t_-}\frac{t\,\d t}{\sqrt{(1-t^2)(t-t_-)(t-t_+)}},
\\
I=\frac4\pi\int_{t_+}^1\frac{\d t}{\sqrt{(1-t^2)(t-t_-)(t-t_+)}}
=\frac4\pi\int_{-1}^{t_-}\frac{\d t}{\sqrt{(1-t^2)(t-t_-)(t-t_+)}}.
\end{gather*}
The latter equality is the standard relation of the two real periods on an elliptic curve
and we also have
\begin{equation}
J_--J_+=4
\label{eq2}
\end{equation}
from differentiating identity \eqref{eqpm} with respect to~$a$.

Letting $k=4(a^2-1)/(a^2+1)$, we have
$$
\frac k2=\frac{2(a^2-1)}{a^2+1}
=\frac{8u(u^2-1)}{(u^2+1)^2},
$$
implying, on the basis of~\eqref{m-},
$$
\frac{\partial\m(P_{1,k})}{\partial u}
=-\frac{u^4-6u^2+1}{(u^2+1)^3}K,
$$
where
$$
K=\frac8\pi\int_{1-k/2}^1\frac{\d T}{\sqrt{(1-T^2)(T+k/2-1)(T+k/2+1)}}.
$$

To establish \eqref{eq+-} we need to prove that the derivatives of the both sides agree:
\begin{equation}
\frac{u^4-6u^2+1}{(u^2+1)^3}K
=\frac{(u^2+1)^3(J_-+3J_+)+4(u^2+2u-1)(u^4-2u^3+2u^2+2u+1)I}{(u^2+1)^2(u^4-6u^2+1)},
\label{eq1}
\end{equation}
while the latter equality follows from
\begin{align}
K&=\frac{2(u^2+1)}{u^2+2u-1}\,I,
\label{eq4}
\\
J_-+3J_+ &=-\frac{2(u^2+2u-1)}{u^2+1}\,I.
\label{eq3}
\end{align}

By switching the parameter $u>\sqrt2+1$ to $v=(u^2+2u-1)/(u^2+1)$, which ranges from $1$ to~$\sqrt2$, we can write
$$
(t-t_-)(t-t_+)
=t^2+\frac{2(v^2-1)}v\,t+\frac{v^4-v^2-1}{v^2},
\qquad
\frac k2=2(v^2-1),
$$
so that
\begin{align*}
K&=\frac{8}{\pi}\int_{3-2v^2}^1\frac{\d T}{\sqrt{(1-T^2)(T+2v^2-1)(T+2v^2-3)}},
\\
\frac{u^2+1}{u^2+2u-1}\,I&=\frac{4}{\pi}\int_{(\sqrt{2-v^2}-v^2+1)/v}^1\frac{\d t}{\sqrt{(1-t^2)(v^2t^2+2(v^2-1)vt+v^4-v^2-1)}},
\end{align*}
hence equality \eqref{eq4} follows from Lemma~\ref{lem-EI1}.
The change of the parameter in the integrals defining $J_-$ and $J_+$ together with relation~\eqref{eq2} and Lemma~\ref{lem-EI2}
show that equality \eqref{eq3} is true as well. This establishes the equality of the derivatives \eqref{eq1}
and concludes the proof of Theorem~\ref{th1}.
\end{proof}

\section{Ramanujan's modular parametrization}
\label{modular}

In this section we prove Theorem~\ref{th2}.

A particular entry of Ramanujan's notebooks \cite[p.~236, Entry 68]{Be94} says that
$$
AB+\frac 7{AB}=\biggl(\frac AB\biggr)^2+\biggl(\frac BA\biggr)^2-3,
$$
where $A=A(\tau)=\eta(\tau)/\eta(7\tau)$ and $B=B(\tau)=A(3\tau)$, and $\eta(\tau)$ denotes the Dedekind eta function.
(We believe that this other eta notation does not cause any confusion with~\eqref{etadef} as it depends here on a single variable.)
We can interpret Ramanujan's equation as the modular-unit parametrization
\begin{equation}
\label{unit1-2}
\begin{aligned}
x=\wt x(\tau)
&=\frac1{\sqrt7}\,\frac{\eta(\tau)\eta(3\tau)}{\eta(7\tau)\eta(21\tau)}
=\frac{1}{\sqrt{7}}g_1(\tau)g_2(\tau)g_3^2(\tau)g_4(\tau)g_5(\tau)g_6^2(\tau)g_8(\tau)g_9^2(\tau)g_{10}(\tau),
\\
y=\wt y(\tau)
&=-\biggl(\frac{\eta(\tau)\eta(21\tau)}{\eta(3\tau)\eta(7\tau)}\biggr)^2
=-(g_1(\tau)g_2(\tau)g_4(\tau)g_5(\tau)g_8(\tau)g_{10}(\tau))^2
\end{aligned}
\end{equation}
of the elliptic curve $P_{\sqrt7,3}(x,y)=0$ of conductor~21, where
\begin{equation*}
g_a(\tau)= q^{21B_2(a/21)/2}\prod_{\substack{n \geq 1 \\ n \equiv a \bmod 21}}(1-q^n)\prod_{\substack{n \geq 1 \\ n \equiv -a \bmod 21}}(1-q^n), \qquad q= e^{2\pi i \tau},
\end{equation*}
are level~21 modular units, and $B_2(x)= x^2-x+1/6$ is the second Bernoulli polynomial.

From here on, we let
\begin{equation*}
x_0(\tau)=\sqrt{7}\,\wt x(\tau)=g_1(\tau)g_2(\tau)g_3^2(\tau)g_4(\tau)g_5(\tau)g_6^2(\tau)g_8(\tau)g_9^2(\tau)g_{10}(\tau).
\end{equation*}

The rest of this section is devoted to proving Theorem~\ref{th2}.

\begin{lemma}
\label{lem2/7}
The following identity is true:
\begin{equation*}
\m^-(P_{\sqrt7,3})=\frac{1}{2\pi}\int_{2/7}^{i\infty}\eta(x_0(\tau),\wt y(\tau)).
\end{equation*}
\end{lemma}

\begin{proof}
Note that for the four CM points on the upper half-plane,
$$
\tau_\pm=\frac{\mp 9+\sqrt{-3}}{42}
\quad\text{and}\quad
\tau_\pm'=\frac{\mp 9+\sqrt{-3}}{21},
$$
we have
$$
(\wt x(\tau_\pm),\wt y(\tau_\pm))=\bar S_\pm
\quad\text{and}\quad
(\wt x(\tau_\pm'),\wt y(\tau_\pm'))=\bar T_\pm,
$$
where the latter points \eqref{pointsST} lie on the elliptic curve $P_{\sqrt7,3}(x,y)=0$.

The Atkin--Lehner involution
\[W_{21}=\frac{1}{\sqrt{21}}\left(\begin{array}{cc} 0&-1\\21&0 \end{array}\right)\]
acts on the modular functions as follows \cite[Corollary~2.2]{CL98}:
$\wt x(\tau)\mapsto 1/\wt x(\tau)$ and $\wt y(\tau) \mapsto \wt y(\tau)$.
In addition, $W_{21}\tau_\pm=\tau_\mp$ and $W_{21}\tau_\pm'=\frac14\tau_\mp'$.

Consider the geodesic $[\tau_-,\tau_+]$ in $X_0(21)$. All the points in this path
have norm $1/21$. Therefore, the involution $W_{21}$ sends each point $\tau$ to $-1/(21\tau)=-\ol\tau$ on the geodesic and also
reverses the orientation. In particular, the image of $(\wt x,\wt y)$ along the path $[\tau_-,\tau_+]$
is the complex conjugated curve $(\ol{\wt x},\ol{\wt y})$.
On the other hand, $W_{21}$ sends $(\wt x,\wt y)$ to $(1/\wt x,\wt y)$ meaning that
we have $|\wt x|=1$ and $\wt y$ real along the geodesic $[\tau_-,\tau_+]$ in $X_0(21)$.
This implies that the path corresponds to $[S_-,S_+]$ on the elliptic curve.
Evaluating $\wt y$ in the middle point of the geodesic
$\tau=i/\sqrt{21}$ gives a number of absolute value less than~1 and
shows that $\wt y$ corresponds to $y_+$ (that is, to $1/y_-$) along the path.

In other words, we can write
\begin{align*}
\m^-(P_{\sqrt7,3})
&=\frac{1}{2\pi i}\int\limits_{|x|=1}\log^+|y_-(x)|\frac{\d x}{x}
=-\frac{1}{2\pi}\int_{[S_-,S_+]}\eta(x,y_-)
\\
&=\frac{1}{2\pi}\int_{[\tau_-,\tau_+]} \eta(\wt x(\tau), \wt y(\tau))
=\frac{1}{2\pi}\int_{[\tau_-,\tau_+]} \eta(x_0(\tau), \wt y(\tau)),
\end{align*}
where in the last manipulation we use the fact that $\wt y$ is real along the path.

Our goal is to integrate by choosing paths that avoid the singularities with non-trivial residue.
Let us first compute the tame symbol at each point involved in our integrals. This can be done directly from
the tame symbol computations in Lemma~\ref{tame}.
The relationship between the tame symbols at a point $R$ is given by
\begin{align*}
(x_0,\wt y)_R
&\equiv \frac{x_0^{v_R(\wt y)}}{\wt y^{v_R(x_0)}}\bmod \mathcal{M}_R
\equiv \frac{(\sqrt{7}x)^{v_R(y)}}{y^{v_R(x)}}\bmod \mathcal{M}_R
\\
&= \sqrt{7}^{v_R(y)}(x,y)_R,
\end{align*}
where an ambiguity in setting $\wt y(\tau)=y$ is compensated later by making the correspondence between points in the Weierstrass form and values of~$\tau$.
Thus,
\[
|(x_0,\wt y)_P|=\frac{1}{7}, \quad  |(x_0,\wt y)_{-Q}|=1,\quad  |(x_0,\wt y)_{P+Q}|=1,\quad |(x_0,\wt y)_O|=7.
\]
Similarly,
\[
|(x_0/7,\wt y)_P|=1, \quad |(x_0/7,\wt y)_{-Q}|=\frac{1}{7},\quad |(x_0/7,\wt y)_{P+Q}|=7, \quad |(x_0/7,\wt y)_O|=1.
\]

Note that $\tau=i\infty$ yields $\wt x(\tau)\to \infty$ and $\wt y(\tau) \to 0$. Equating $\wt y(\tau)$ to $y$ we obtain that $\tau=i\infty$ corresponds to
$-Q$. Similarly, $\tau=0$ yields $O$, and $\tau=2/7$ yields $P+Q$. Thus, there are nontrivial residue singularities for $\eta(x_0(\tau),\wt y(\tau))$ at $\tau=0$ and similarly for
$\eta(x_0(\tau)/7,\wt y(\tau))$ at $\tau=i\infty$ and~$2/7$.

The other Atkin--Lehner involution
$$
W_7=\frac1{\sqrt7}\begin{pmatrix} 7 & 2 \\ 21 & 7 \end{pmatrix}
$$
acts on the modular functions as follows \cite[Corollary~2.2]{CL98}:
$\wt x(\tau)\mapsto1/\wt x(\tau)$, $x_0(\tau)\mapsto7/x_0(\tau)$, and $\wt y(\tau)\mapsto1/\wt y(\tau)$,
so that
$$
\eta(x_0(W_7\tau),\wt y(W_7\tau))=\eta(x_0(\tau)/7,\wt y(\tau)).
$$
Furthermore, $W_7\tau_+=\tau_-$ and $W_7\tau_+'=\tau_-'$.

We remark that $W_7(0)=2/7$. Therefore, we have
\begin{align*}
\int_{[\tau_-,\tau_+]}\eta(x_0(\tau),\wt y(\tau))
&=\int_{\tau_-}^{2/7}\eta(x_0(\tau),\wt y(\tau))+\int_{2/7}^{\tau_+}\eta(x_0(\tau),\wt y(\tau))
\\
&=\int_{W_7^{-1}\tau_-}^{W_7^{-1}(2/7)}\eta(x_0(W_7(\tau)),\wt y(W_7(\tau)))+\int^{\tau_+}_{2/7}\eta(x_0(\tau),\wt y(\tau))
\\
&=\int_{\tau_+}^{0}\eta(x_0(\tau)/7,\wt y(\tau))+\int^{\tau_+}_{2/7}\eta(x_0(\tau),\wt y(\tau))
\\
&=\int_{\tau_+}^{i/\sqrt{21}}\eta(x_0(\tau)/7,\wt y(\tau))+\int_{i/\sqrt{21}}^0\eta(x_0(\tau)/7,\wt y(\tau))+\int_{2/7}^{\tau_+}\eta(x_0(\tau),\wt y(\tau)).
\end{align*}
Observe that $\wt y$ is real on the geodesic $[\tau_+,i/\sqrt{21}]$,
and the same is true on the imaginary axis $i\mathbb{R}_+$. Moreover,
$x_0$ and $x_0/7$ are also real on $i\mathbb{R}_+$. Thus,
\[
\int_{\tau_+}^{i/\sqrt{21}}\eta(x_0(\tau)/7,\wt y(\tau))
= \int_{\tau_+}^{i/\sqrt{21}}\eta(x_0(\tau),\wt y(\tau)),
\quad
\int_{i/\sqrt{21}}^0\eta(x_0(\tau)/7,\wt y(\tau))=0,
\]
and
\[
\int_{i/\sqrt{21}}^{i\infty} \eta(x_0(\tau),\wt y(\tau))=0.
\]
Continuing the earlier computation we obtain
\begin{align*}
\int_{\tau_-}^{\tau_+}\eta(x_0(\tau),\wt y(\tau))
&=\int_{\tau_+}^{i/\sqrt{21}}\eta(x_0(\tau),\wt y(\tau))+\int_{i/\sqrt{21}}^{i\infty}\eta(x_0(\tau),\wt y(\tau))+\int_{2/7}^{\tau_+}\eta(x_0(\tau),\wt y(\tau))
\\
&=\int_{2/7}^{i\infty}\eta(x_0(\tau),\wt y(\tau)),
\end{align*}
as desired.
\end{proof}

\begin{lemma}
\label{lem38log7}
We have
\begin{equation*}
\frac{1}{2\pi}\int_{2/7}^{i\infty}\eta(x_0(\tau),\wt y(\tau)) = \frac{1}{2}L'(f_{21},0)+\frac{3}{8}\log 7.
\end{equation*}
\end{lemma}

\begin{proof}
By equations \eqref{unit1-2} and \cite[Theorem 1]{Zu14} we have
\begin{equation}\label{E:Int}
\int_{2/7}^{i\infty}\eta(x_0(\tau),\wt y(\tau)) = \frac{1}{\pi}L(f,2),
\end{equation}
where
\begin{align*}
f(\tau)
&=12q+15q^2+12q^3+42q^4+\cdots
\\
&=\frac{21}{4}f_{21}(\tau)+\frac{9}{32}(96-E_2(\tau)+3E_2(3\tau)+49E_2(7\tau)-147E_2(21\tau))
\end{align*}
and
\begin{equation*}
E_2(\tau)=1-24 \sum_{k,l>0}k q^{kl}
\end{equation*}
is the (normalized) logarithmic derivative of the Dedekind eta function.

Let us compute $L(g,s)$, where $g(\tau)=96-E_2(\tau)+3E_2(3\tau)+49E_2(7\tau)-147E_2(21\tau)$.
It follows immediately from the definition of $E_2(m\tau)$ that
\begin{align*}
L(g,s)
&=-24\left(-1+\frac{3}{3^s}+\frac{49}{7^s}-\frac{147}{21^s}\right)\sum_{k,l>0}\frac{1}{k^{s-1}l^s}
\\
&= -24\left(-1+\frac{3}{3^s}+\frac{49}{7^s}-\frac{147}{21^s}\right)\zeta(s-1)\zeta(s).
\end{align*}
Computing the related Laurent series expansions we get
\begin{equation*}
-1+\frac{3}{3^s}+\frac{49}{7^s}-\frac{147}{21^s} \sim -\frac{2\log 7}{3}(s-2)
\quad\text{and}\quad
\zeta(s-1)\sim\frac{1}{s-2}
\end{equation*}
as $s\to2$. Hence
\begin{equation}\label{E:log7}
L(g,2)=\frac{\pi^2}{6}\lim_{s\to2}\left(-24\left(-1+\frac{3}{3^s}+\frac{49}{7^s}-\frac{147}{21^s}\right)\zeta(s-1)\right)
=\frac{8\pi^2}{3}\log 7.
\end{equation}
Finally, plugging \eqref{E:log7} into \eqref{E:Int} yields
\begin{equation*}
\frac{1}{2\pi}\int_{2/7}^{i\infty}\eta(x_0(\tau),\wt y(\tau))
= \frac{21}{8\pi^2}L(f_{21},2)+\frac{3}{8}\log 7 = \frac{1}{2}L'(f_{21},0)+\frac{3}{8}\log 7,
\end{equation*}
as required.
\end{proof}

Combining the results of Lemmas~\ref{lem2/7} and \ref{lem38log7} we arrive at the conclusion of Theorem~\ref{th2}.
This leads to the evaluation~\eqref{Emain}.

\section{Conclusion}
\label{final}

Many other cases exist, not covered by our Theorem~\ref{th-log}, where the Mahler measures $\m(P_{a,b,c})$ can be identified
as $\Q$-linear combinations of the corresponding $L$-value and logarithm. Many other instances of similar evaluations
of the ``half-Mahler'' measures $\m^-(P_{a,c})$ and $\m^+(P_{a,c})$ show up, not covered by our Theorem~\ref{th1}.
Such examples are generated on demand by numerical computation, and the majority of them sound pretty ugly (in terms
of coefficients in the rational combination). There are however some interesting subfamilies that are worth mentioning
as they can be potentially useful in establishing other conjectural evaluations of Mahler measures.

First, Theorem~\ref{th1} can be generalized further to include $k>4$. Besides the numerical evidence for that,
this is strongly supported by the claims in Sections~\ref{Weierstrass}--\ref{isogeny}, in particular, by Proposition~\ref{prop:regul}\,---\,the results
independent of the constraint $0<k<4$. On the other hand, the auxiliary statements in Section~\ref{diff} must be revisited for $k>4$.

Second, the special case~\eqref{torsion6} noticed in Section~\ref{Weierstrass} corresponds to an isomorphism
between the curve $P_{a,a^2-1}(x,y)=0$ and a curve from another family of Boyd in~\cite{Bo98}, namely, the elliptic curve
\[
(1+x')(1+y')(x'+y')-(a^2-1)x'y'=0.
\]
It is given by
\[
x'=\frac{x(ax+y+a^2-1)}{x+ay}, \quad y'=-\frac{x(ax+y)}{x+ay}.
\]
The isomorphism corresponds to the following simple relation between the corresponding (full) Mahler measures: for every $a\geq 1$
\begin{equation*}
\m\bigl((1+x)(1+y)(x+y)-(a^2-1)xy\bigr) = \frac32\,\m(P_{a,a^2-1})-\log a,
\end{equation*}
and this can be shown by the techniques similar to that in Section~\ref{diff}. Note that $\m^+(P_{a,a^2-1})=0$, hence
$\m^-(P_{a,a^2-1})=\m(P_{a,a^2-1})$ in this case.

Furthermore, we would like to mention that Boyd himself had the idea of an intelligent split of the Mahler measure
of a 3-variate polynomial into pieces. The details of this story, including the original conjectures and partial results
on their resolution, can be found in~\cite{Lalin-Crelle}.

Finally, we believe that the results in this paper substantiate the importance of study of the associated Mahler measures
of parametric families that do not happen to be tempered.

\medskip
\textbf{Acknowledgements.}
The principal part of the work was done during the
thematic year 2014--2015 ``Number Theory, from Arithmetic Statistics to Zeta Elements''
at the Centre de recherches math\'ematiques, Universit\'e de Montr\'eal.
We thank the staff of the institute for the excellent conditions we experienced when conducting this research.
We are also grateful to the anonymous referee whose feedback helped us improve on the presentation of the paper.

The first-named author was supported by Discovery Grant 355412-2013 of the Natural Sciences and Engineering Research Council of Canada
and by the Subvention \'etablissement de nouveaux chercheurs 144987 of the Fonds de recherche du Qu\'ebec\,---\,Nature et technologies.
The second-named author was supported by a Canada Research Chair grant.
The third-named author was supported by Discovery Project 140101186 of the Australian Research Council.

\end{document}